\documentclass[12pt]{amsart}
\usepackage{epsfig,color}

\makeatother

\theoremstyle{definition}

\def\fnum{equation} 
\newtheorem{Thm}[\fnum]{Theorem}
\newtheorem{Cor}[\fnum]{Corollary}

\newtheorem{Lem}[\fnum]{Lemma}
\newtheorem{Con}[\fnum]{Conjecture}

\newtheorem{Pro}[\fnum]{Proposition}

\numberwithin{equation}{section}

\newcommand{\R}{{\text{R}}}

\newcommand{\nn}{{\bf{n}}}
\newcommand{\Ric}{{\text{Ric}}}

\newcommand{\Hess}{{\text {Hess}}}
\def\ZZ{{\bold Z}}
\def\RR{{\bold R}}

\def\SS{{\bold S}}

\def\CC{{\bold C }}

\newcommand{\e}{{\text {e}}}

\newcommand{\Area}{{\text {Area}}}

\def\bH{{\bold H}}
\def\bN{{\bold N}}
\def\bB{{\bold B}}

\newcommand{\cA}{{\mathcal{A}}}

\newcommand{\cF}{{\mathcal{F}}}
\newcommand{\cL}{{\mathcal{L}}}

\newcommand{\cS}{{\mathcal{S}}}

\newcommand{\cV}{{\mathcal{V}}}

\newcommand{\eqr}[1]{(\ref{#1})}

\title[Entropy and codimension bounds for generic singularities]{Entropy and codimension bounds for generic singularities}

\author{Tobias Holck Colding}%
\address{MIT, Dept. of Math.\\
77 Massachusetts Avenue, Cambridge, MA 02139-4307.}
\author{William P. Minicozzi II}%

\thanks{The  authors
were partially supported by NSF Grants DMS 1812142 and DMS 1707270.}

\email{colding@math.mit.edu and minicozz@math.mit.edu}

\begin{document}

\maketitle

\begin{abstract}
We show that all closed $2$-dimensional singularities for higher codimension mean curvature flow that cannot be perturbed away have uniform entropy bounds and lie in a linear subspace of small dimension.  The entropy and dimension of the subspace are both $\leq C\,(1+\gamma)$ for some universal constant $C$ and genus $\gamma$. 
These are the first general bounds on generic singularities in arbitrary codimension.
\end{abstract}

 \section{Introduction}
Even for hypersurfaces,  examples show that singularities of mean curvature flow (MCF) are too numerous to classify.     The hope is that the generic ones that cannot be perturbed away are much simpler.   
Indeed for hypersurfaces in all dimensions generic singularities have  been classified in \cite{CM3}.  These are round generalized cylinders $\SS^k_{\sqrt{2\,k}}\times \RR^{n-k}$.   

Higher codimension  MCF   is a complicated nonlinear parabolic system where much less is known.    
 Singularities are modeled by shrinkers that evolve by scaling, \cite{H}, \cite{I}, \cite{W}.             We show that the only closed $2$-dimensional generic singularities, i.e.,   $F$-stable shrinkers, have a uniform entropy bound and lie   in a small linear subspace.  The entropy and dimension of the subspace are both $\leq C\,(1+\gamma)$ for a universal constant $C$ and genus $\gamma$. 

The entropy is a Lyapunov function for the flow that is particularly useful.  To define it, recall that the Gaussian surface area $F$ of an $n$-dimensional submanifold
 $\Sigma^n \subset \RR^N$ is
\begin{align}
	F(\Sigma) = \left( 4\,\pi \right)^{ - \frac{n}{2}} \,  \int_{\Sigma} \e^{ - \frac{|x|^2}{4} } \, .
\end{align}
Following \cite{CM3}, the entropy $\lambda$ is the supremum of $F$ over all translations and dilations
\begin{align}    
	 \lambda (\Sigma) = \sup_{c,x_0} \, F (c\,\Sigma + x_0)
\end{align}
For a shrinker, the entropy is equal to the $F$-functional.  
   By Huisken's monotonicity, \cite{H}, it follows that $\lambda$ is monotone nonincreasing under the flow.  From this, and lower semi continuity of $\lambda$, 
all blowups have entropy bounded by that of the initial submanifold.   

An immersed submanifold $\Sigma^n \subset \RR^{N}$ is a  {\emph{shrinker}} if  the mean curvature ${\bf{H}} = \frac{1}{2} \, x^{\perp}$.

\begin{Thm}	\label{t:entropyextra2}
There exists a universal constant $C$ so that
if $\Sigma^2 \subset \RR^N$ is a closed $F$-stable shrinker of genus $g$ and $N \geq C \, \lambda (\Sigma)$, then $\Sigma\subset \cV$ where $ \cV$ is a linear subspace and 
\begin{align}
\lambda (\Sigma) &\leq C \, (1+\gamma)\, ,	\label{e:entro1} \\
\dim \,  \cV &\leq C \, (1+\gamma)\, .\label{e:entro2}
\end{align}
\end{Thm}

This gives the first general bounds on generic singularities of surfaces in arbitrary codimension.
When $\Sigma$ is diffeomorphic to a sphere, \eqr{e:entro1} becomes
\begin{align}	\label{e:becomes}
	\lambda (\Sigma) &<  4 = \e \, \lambda (\SS^2_2) \, .
\end{align}
The sharp constant is unknown, but \eqr{e:becomes} is at most off by a 
  factor of   $\e$.
Theorem \ref{t:entropyextra2} holds even when the $F$-index is not zero, with $C$   depending on the index. 

There is no analog of \eqr{e:entro1} for minimal surfaces   in $\RR^4$.  Namely, viewing $\RR^4$ as $\CC^2$ one sees that for each integer $m$  the parametrized complex submanifold $z\to (z,z^m)$ is a stable minimal variety that is topologically a plane.  It has $\Area (B_r\cap \Sigma)\geq C\,m\,r^2$ for $r\geq 1$.   In contrast, Theorem \ref{t:entropyextra2} implies that  $\Area (B_r\cap\Sigma)\leq C\,(1+\gamma)\,r^2$ for a closed stable $2$-dimensional shrinker  $\Sigma$  of genus $\gamma$.   Similarly, there is no analog of \eqr{e:entro2} for minimal surfaces.   Indeed, for each $m$, the parametrized surface $z\to (z,z^2,z^3,\cdots,z^{m+1})$ is a stable minimal variety that is topologically a plane.  Its real codimension is $2\,m$ and it is not contained in a proper subspace.   
In contrast to \eqr{e:entro1} and \eqr{e:entro2}, very little is known about stable minimal surfaces in higher codimension.  A notable exception is a result of Micallef, \cite{Mi}, that    a   stable oriented parabolic minimal surface   in $\RR^4$ is complex for some orthogonal complex structure.

\vskip1mm
Entropy is bounded from below by the Gaussian Willmore functional, see Lemma \ref{l:UB}.   
We will also prove a sharp upper bound for the Gaussian Willmore functional $W$ 
\begin{align}
	W(\Sigma)=(4\,\pi)^{-\frac{n}{2}}\int_{\Sigma} |\bH|^2 \, \e^{-\frac{|x|^2}{4}} \, .
\end{align}
The next theorem gives a sharp bound for $W$, in arbitrary codimension,  for stable shrinkers that are topological spheres.

\vskip2mm
\begin{Thm}	\label{t:stableWillmoreA}
If   $\Sigma^2 \subset \RR^N$ is a $F$-stable  shrinker diffeomorphic to a sphere, then 
$W (\Sigma) \leq W(\SS^2_2)$.   With equality if and only if $ \Sigma  = \SS^2_2 \subset \RR^3$ up to rotation. 
\end{Thm}

 We will also prove $W$ bounds for  surfaces of any genus, see Theorem \ref{t:stableWillmore}, in addition to  several other entropy and eigenvalue bounds.

Throughout, $\Sigma^n \subset \RR^N$ will be an immersed shrinker with finite entropy.
 Because of the lack of the maximum principle in higher codimension, embeddedness is not preserved and, thus, is not natural to assume.  
 Shrinking curves are automatically planar and the only $F$-stable ones are lines and circles by Corollary \ref{c:shricur}.

\section{Shrinkers}

Set $f= \frac{|x|^2}{4}$ and define the drift Laplacian, cf. \cite{CM3}, $\cL$ for functions $u$  and tensors by
\begin{align}
\cL\,u=\Delta\,u-\nabla_{ \nabla^T f} u \, .
\end{align}
Then $\cL$ is self-adjoint with respect to the Gaussian $L^2$ norm $\| u \|_{L^2}^2 = (4\pi)^{- \frac{n}{2}} \int |u|^2 \,\e^{-f}$, cf. \cite{CM3}, so that eigenfunctions for distinct eigenvalues are $L^2$-orthogonal.   Here $\Delta$ and $\nabla f$ are the Laplacian and gradient on $\Sigma$.  Since $\nabla^T  f=\frac{x^T}{2}$ it follows that $\cL\,u=\Delta\,u-\frac{1}{2}\, \nabla_{ x^T}  u $.

Shrinkers are characterized variationally as critical points of the Gaussian area $F$.  The shrinker equation is $\bH \equiv \sum_{i=1}^n A_{ii} = \frac{1}{2} \, x^{\perp}$, where $e_i$ is an orthonormal frame for $\Sigma$ and the second fundamental form is given by $A_{ij} = A(e_i , e_j) = \nabla^{\perp}_{e_i} e_j$.
Following \cite{CM3} and \cite{CM6},  define the second-variation operator $L$  by 
\begin{align}
	L = \cL  + \frac{1}{2} + \sum_{k,\ell}  \,  \langle  \cdot   , A_{k \ell}  \rangle \, A_{ k \ell} \, .
\end{align}
Note that $L$ is symmetric with respect to the Gaussian inner product on normal vector fields.
The second  variation  in the normal direction   $u$ is  (\cite{CM3}, \cite{CM6}, \cite{AHW}, \cite{AS}, \cite{LL})
\begin{align}	\label{e:2varr}
\delta^2 (u)=- (4\,\pi)^{-\frac{n}{2}} \, \int \langle u,L\,u\rangle\,\e^{-f}\, .
\end{align}
The second variation is   negative along translations and dilations, so there are no stable shrinkers in the usual sense.  As in \cite{CM3}, a shrinker is said to be  {\it{$F$-stable}} if the second variation  is nonnegative perpendicular to these     unstable directions.  A shrinker is {\it{entropy-stable}} if it is a local minimum for the entropy $\lambda$. Entropy-unstable shrinkers are singularities that can be perturbed away, whereas entropy-stable ones cannot; see \cite{CM3}, \cite{CM7}.   By section $7$ in \cite{CM3}, entropy-stable and $F$-stable are equivalent for closed shrinkers.
By \cite{CM3}, spheres and planes are the only $F$-stable hypersurfaces.  (This was generalized to higher codimension when $\bH$ does not vanish and the principal normal is parallel in \cite{AHW}; see also \cite{AS}, \cite{LL}.)    
It is easy to see that spheres and planes are $F$-stable in any codimension.

There are several ways to show that a shrinker $\Sigma$ is $F$-unstable.  The first, essentially the definition, is to find $u \in L^2$   with $\delta^2 (u) < 0$ that is $L^2$-orthogonal to $\bH$ and translations.  For instance,    $u \in L^2$  with $L \, u = \mu \, u$ with $\mu > 1$  implies $F$-instability.  The second is to find  $u \in L^2$  orthogonal to $\bH$ and ``below the translations'', i.e., with $\int \langle u , L \, u \rangle \e^{-f} > \frac{1}{2} \int |u|^2 \, \e^{-f}$.     In codimension one, $L$ becomes an operator on functions and the lowest eigenfunction does not vanish.  In \cite{CM3}, we used this to conclude that $F$-stability implied mean convexity for hypersurfaces.  Because of the vector-valued nature of things, there is no analog of this in higher codimension unless one assumes that the principal normal is parallel (see \cite{AHW}).

\subsection{Simons type equations for $A$ and translations}

One of the important tools in \cite{CM3}, \cite{CM4} and \cite{CM6} was a series of elliptic equations for various geometric objects on a shrinker, including the second fundamental form, mean curvature and translation vector fields.
Namely, 
if   $V^{\perp}$ is the normal part of  $V \in \RR^N$, then 
\begin{align}
	\left(L \, A \right)_{ij} &=     A_{ij}     
	+     2\,  \sum_{k, \ell}   \,  \langle A_{j \ell}   , A_{ik}  \rangle \, A_{\ell k}    -    \sum_{m,\ell}   \,
	 \left\{ \langle A_{m \ell}  , A_{i \ell}  \rangle A_{jm}   + \langle A_{j \ell}   , A_{m \ell}  \rangle A_{mi}  \right\}   
          \, , \\
          L \, {\bf{H}} &= {\bf{H}}  {\text{ and }}
          L \, V^{\perp}  = \frac{1}{2} \,  V^{\perp} \, .   \label{e:1point5}
\end{align}
One consequence of \eqr{e:1point5} and symmetry of $L$ is that $\bH$ and $V^{\perp}$ are orthogonal with respect to the Gaussian inner product.
These are proved for hypersurfaces in theorem $5.2$ and lemma $10.8$ in \cite{CM3}.
For higher codimension, see proposition $3.6$ in \cite{CM6},  \cite{AHW}, \cite{AS}, or \cite{LL}.

\begin{Lem}	\label{l:dH}
(cf. ($5.6$), ($5.11$) in \cite{CM3})
The derivatives of $\bH$ and   $V^{\perp}$ are
\begin{align}
	\nabla  \, \bH   =  - \langle \bH , A(\cdot , \cdot)  \rangle - \frac{1}{2} \, A (x^T, \cdot ) {\text{ and }}
	\nabla^{\perp} \, V^{\perp}   = - A (\cdot , V^T) \, .
\end{align}
\end{Lem}

\begin{proof}
Let $\e_j$ be an orthonormal frame for $\Sigma$ and differentiate the shrinker equation
\begin{align}
	2\, \nabla_{\e_i} \bH &=  \nabla_{\e_i} \, x^{\perp} = \nabla_{\e_i} \left( x - \langle x , \e_j \rangle \e_j \right) = \e_i - \langle \e_i , \e_j \rangle \e_j 
	- \langle x , \nabla_{\e_i} \e_j \rangle \e_j - \langle x , \e_j \rangle  \nabla_{\e_i} \e_j \notag \\
	&=  
	- \langle x , \nabla_{\e_i} \e_j \rangle \e_j - \langle x , \e_j \rangle  \nabla_{\e_i} \e_j \, .
\end{align}
Now fix a point $p$ and choose the frame $\e_i$ so that $\nabla_{\e_i}^T\e_j = 0$ at $p$.  It follows that (at $p$)
$
	 \nabla_{\e_i} \e_j =  \nabla_{\e_i}^{\perp} \e_j = A (\e_i , \e_j) \, ,
$
so we get that
\begin{align}
	2\, \nabla_{\e_i} \bH &=   
	- \langle x , A(\e_i , \e_j) \rangle \e_j - \langle x , \e_j \rangle  A(\e_i , \e_j) = 
	  -2\,  \langle \bH  , A(\e_i , \cdot) \rangle   -   A(\e_i , x^T) \, .
\end{align}
The first claim follows.  Next, we have
$
	\nabla_{e_i}^{\perp}  \, V^{\perp}  = - \nabla_{e_i}^{\perp}  \, V^{T} = - A(e_i , V^T) $.
	\end{proof}

\begin{Lem}	\label{l:phiVvar}
Given a function $\phi$ and $V \in \RR^N$, the second variation for $\phi \, V^{\perp}$ and $\phi\,\bH$ are  
\begin{align}
	\delta^2 (\phi \, V^{\perp}) &= (4\,\pi)^{-\frac{n}{2}} \, \int \left[ |\nabla \phi |^2 -  \frac{1}{2} \, \phi^2 \right] \, \left| V^{\perp} \right|^2 \, \e^{-f} \, .\label{e:phiVvar}\\
	\delta^2 (\phi \, \bH) &=  (4\,\pi)^{-\frac{n}{2}} \, \int \left[ |\nabla \phi |^2 - \phi^2 \right] \, \left| \bH \right|^2 \, \e^{-f} \, .\label{e:phiHvar}
\end{align}
\end{Lem}

\begin{proof}
 Given any normal section $u$, 
the Leibniz rule $\cL (\phi\,u) = \phi\, \cL\, u + (\cL \,\phi)\, u + 2\,\nabla^{\perp}_{\nabla \phi} u$ gives 
 \begin{align}	\label{e:leibnVvar}
 	(4\,\pi)^{\frac{n}{2}} \, \delta^2 (\phi\,u) &= - \int \langle  \phi\, u , L (\phi\,u) \rangle\,\e^{-f}  = - \int \left(\phi^2 \, \langle  u  , L\, u \rangle + (\phi\, \cL \phi)\, |u|^2 + 2\,\langle \phi\, u , \nabla^{\perp}_{\nabla \phi} u \rangle \right)\,\e^{-f}\notag \\
	&=  - \int
	 \left(\phi^2 \, \langle  u  , L\, u \rangle + (\phi\,\cL \phi)\, |u|^2 +  \langle \phi\,\nabla \phi , \nabla |u|^2 \rangle\right)\,\e^{-f}\\  
	  &=   \int \left( |\nabla \phi|^2 \, |u|^2-\phi^2 \, \langle  u  , L\, u \rangle\right)\,\e^{-f}     \, ,\notag
 \end{align}
 where the last equality used integration by parts.
 The claims follow  from applying \eqr{e:1point5} and \eqr{e:leibnVvar} with $u = V^{\perp}$ and $u=\bH$.   
\end{proof}

 \section{The Hessian equation}
 
 We will see that shrinkers satisfy a Hessian equation.    Define the symmetric $2$-tensor $A^{\bH} = \langle A , \bH \rangle$ and define a symmetric operator $A^2$ on tangent vector fields by
 \begin{align}
 	\langle e_ i , A^2 (e_j) \rangle \equiv  \langle A_{ik} , A_{kj} \rangle \, .
 \end{align}
 
 \begin{Pro}	\label{p:shrink}
 If $\Gamma^n \subset \RR^N $ is a shrinker or an $n$-plane, then 
  \begin{align}  \label{e:shrink}
 	  \Hess^{\Gamma}_{ f } -A^{\bf{H}}  = \frac{1}{2} \, \langle \cdot , \cdot \rangle   \, . 
 \end{align}
For hypersurfaces, the converse also holds.
 \end{Pro}
 
 It is interesting to compare this   with shrinking solitons for the Ricci flow,   \cite{Ha}.  A  gradient  shrinking soliton   is a manifold $M$,  metric $g$ and   function $f$  satisfying
$
 \Hess_f+\Ric=\frac{1}{2}\,g$.

 \begin{Lem}	\label{l:hessx2}
 If $\Gamma^n \subset  \RR^{N}$, then
$
 	\Hess^{\Gamma}_{|x|^2} = 2 \, \langle \cdot  , \cdot \rangle + 2\langle x^{\perp} , A  \rangle$.
	\end{Lem}

\begin{proof}
Given an orthonormal frame $e_i$ for $\Gamma$, we compute
\begin{align}
	\frac{1}{2} \, \Hess^{\Gamma}_{|x|^2}  (e_i , e_j) &= \langle \nabla_{e_i} x^T , e_j \rangle = 
	  \langle e_i  -   \nabla_{e_i}  x^{\perp} , e_j \rangle = \delta_{ij} +   \langle \nabla_{e_i} e_j ,  x^{\perp} \rangle \, .
\end{align}
\end{proof}

 \begin{proof}[Proof of Proposition \ref{p:shrink}]
 Equation  \eqr{e:shrink} holds on a shrinker since
   Lemma \ref{l:hessx2} gives 
 \begin{align}
 	  \frac{1}{2} \, \langle \cdot  , \cdot \rangle +  A^{\bf{H}}   - \Hess^{\Gamma}_{f} = \langle   {\bf{H}} - \frac{1}{2} \, x^{\perp} , A \rangle  \, .
 \end{align}
For the converse, suppose that  \eqr{e:shrink} holds and $\Gamma$ is a hypersurface with unit normal $\nn$.  It follows that at every point either $A=0$ or $\Gamma$ satisfies the shrinker equation.  If $A\equiv 0$, then $\Sigma$ is a hyperplane.  When $A$ is not identically zero, then
let 
	$\cS = \left\{ \bH - \frac{1}{2} \, x^{\perp} = 0 \right\}$
 be where $\Gamma$ satisfies the shrinker equation.    This must be nonempty and closed.  We will argue by contradiction to show that $\cS = \Gamma$.  Let $U$ be a component of the (necessarily open) complement of $S$. Note that $U$ is path connected since it is connected and locally path-connected by 
  theorem $25.5$ in \cite{Mu}.
  It follows that $A=0$ on $U$ and, thus, that $\langle x , \nn \rangle$ is constant on $U$.  Since $U$ cannot be all of $\Gamma$ (since $\cS$ is nonempty), there must be a boundary point $p \in \cS \cap \partial U$.  Since the set where $A=0$ is closed, we see that $\bH(p) =0$ and, thus, that $\langle x , \nn \rangle (p) = 0$.    It follows that 
$\langle x , \nn \rangle \equiv 0$ on all of $U$ and, thus, that $\Gamma$ satisfies the shrinker equation in $U$, giving the desired contradiction.
  \end{proof}
  
  \vskip1mm
  The next lemma recalls the standard Gauss equation   for the Ricci curvature $\Ric$ and scalar curvature $S$.  By convention, the Riemann tensor is given in an orthonormal frame $e_j$ by
  \begin{align}	\label{e:riemannR}
  	\R_{ijk\ell} = \langle \nabla_{e_j} \nabla_{e_i} e_k - \nabla_{e_i} \nabla_{e_j} e_k - \nabla_{[e_j , e_i]} e_k , e_{\ell} \rangle \, , 
  \end{align}
   and the Ricci tensor is $\Ric_{ij} = \sum_k \, \R_{kikj}$.
  
    \begin{Lem}	\label{l:gauss}
 If  $\Gamma^n \subset \RR^N $, then 	$\Ric  = - A^2 -  A^{\bH}$ and 	$S  = H^2 - |A|^2$.
    \end{Lem}
  
  \begin{proof}
 The Gauss equation gives $  	\R_{ijkn} = \langle A_{ik} , A_{jn} \rangle - \langle  A_{jk} , A_{in} \rangle$.
   Summing this over $j=n$ and using that $A_{jj} = -\bH$, we get
  \begin{align}
  	\Ric_{ik} = \R_{ijkj} =  \langle A_{ik} , A_{jj} \rangle - \langle A_{jk} , A_{ij} \rangle= -  \, A^{\bH}_{ik} - \left( A^2 \right)_{ik} \, .
  \end{align}
  This gives the first claim.  Taking the trace gives the second claim.
  \end{proof}

 \begin{Cor}
 If $\Sigma^n \subset \RR^N $ is a shrinker, then
 \begin{align}	\label{e:corPhi}
 \Hess_f+\Ric=\frac{1}{2}\, \langle \cdot  , \cdot \rangle -A^2\leq \frac{1}{2}\,\langle \cdot  , \cdot \rangle \, .
 \end{align}
 \end{Cor}
 
 \begin{proof}
 Lemma \ref{l:gauss}  and Proposition \ref{p:shrink}   give  that 
 \begin{align}   \label{e:rewritten}
 	 \Ric+\Hess_f =  - A^2 -  A^{\bH} +\Hess_f  = \frac{1}{2}\, \langle \cdot  , \cdot \rangle -A^2 \, .
\end{align}
 For     a tangent vector $V$, we have  $\langle A^2 (V) , V \rangle    =|A(V)|^2$, giving the inequality.
 \end{proof}

 \subsection{Minimal submanifolds in spheres}

 \begin{Lem}	\label{p:basicH}
For a submanifold $\Gamma^n \subset \partial B_{\sqrt{2n}} \subset \RR^{N}$ the following are equivalent:
\begin{itemize}
\item[(A)] $\Gamma$ is a shrinker in $\RR^{N}$.
\item[(B)] $\Gamma$ is a minimal submanifold of the sphere $\partial B_{\sqrt{2n}} \subset \RR^{N}$.
\item[(C)] $A^{\bH} = - \frac{1}{2} \, \langle \cdot , \cdot \rangle$.

\end{itemize}
\end{Lem}

\begin{proof}
The   equivalence of (A) and (B) is well known.  (A) implies (B) since the $F$ functional is equivalent to area for spherical submanifolds.
Fix a point $p$ in $\Gamma$ and let $e_i$ be an orthonormal frame for $\Gamma$ with $\nabla_{e_i}^T e_j = 0$ at $p$.  Since the $e_i$'s are tangent also to the sphere, we have $\langle e_i , x \rangle = 0$.
Differentiating this gives
\begin{align}	\label{e:nHx}
	- n = \langle \nabla_{e_i} e_i , x \rangle  = -\langle \bH , x \rangle \, .
\end{align}
 In   (B),  $\bH = u \, x$ for a function $u$ on $\Sigma$.  By \eqr{e:nHx}, $u \equiv \frac{1}{2}$, giving (A).  Furthermore, (A)  and 
Proposition \ref{p:shrink} imply (C).
Finally, we will show that (C) implies (A) and (B).  Taking the trace of (C) gives that $|\bH|^2 = \frac{n}{2}$.  Since $|x|^2 \equiv 2\,n$ on $\Gamma$,  $x$ is normal to $\Gamma$ so
\begin{align}
	\langle \bH , x \rangle = - \langle \nabla_{e_i} e_i , x \rangle = \langle e_i , \nabla_{e_i} x \rangle = n \, .
\end{align}
It follows that
\begin{align}
	\left| \bH - \frac{x}{2} \right|^2 = |\bH|^2 + \frac{|x|^2}{4} - \langle \bH , x \rangle = \frac{n}{2} + \frac{2n}{4} - n = 0 \, .
\end{align}
We conclude that $\bH = \frac{x}{2}$, giving (A) and (B) and, thus, completing the proof.
\end{proof}

 \section{Proof of Theorem \ref{t:entropyextra2}}

We say that $u \in L^2$ is an eigenfunction of $\cL$ with eigenvalue $\mu$ if $\cL\,u+\mu\,u=0$.   Let $\mu_0 = 0 < \mu_1 \leq \mu_2 \leq \dots$ be the eigenvalues of $\cL$ on $\Sigma$ and $u_0  , u_1, \dots$ the corresponding $L^2$-orthonormal eigenfunctions (note that $u_0$ is constant); see \cite{CM5} for details.

\begin{Pro}	\label{p:extradims}
If $\Sigma^n \subset \RR^N$ is contained in a proper linear subspace $ \cV \subset \RR^N$ and is $F$-stable, then  $\mu_1 \geq \frac{1}{2}$.
\end{Pro}

\begin{proof}
Let  $\phi$ be an eigenfunction with $\cL \, \phi = - \mu \, \phi$ and $\mu > 0$.   Let $E \in  \cV^{\perp} \subset \RR^N$ be a unit vector.   
Observe that
\begin{align}	\label{e:Lofit}
	L \, (\phi \, E) =\left[  \left(\cL + \frac{1}{2}\right)\, \phi \right] \, E =   \left(  \frac{1}{2} - \mu \right)\, \phi  \, E \, .
\end{align}

We will show that $\phi \, E$ is an allowable variation, i.e., is orthogonal to $\bH$ and all translations.
Since $\Sigma \subset  \cV$, we have at each point that $\bH$ is parallel to $ \cV$ and, thus, that $\langle \bH , E \rangle = 0$ point-wise.  Let $V$ be any vector parallel to $ \cV$ and note that $  V^T$ must also be parallel to $ \cV$.
Thus, 
  $\langle E , V^{\perp} \rangle =0$ point-wise.  Finally, the last translation vector field is $E$ itself and $\int \langle \phi \, E , E \rangle \, \e^{-f} = \int \phi \, \e^{-f} = 0$.
Since $\phi \, E$ is allowable, $F$-stability and \eqr{e:Lofit} give 
\begin{align}
	0 \geq \int \langle \phi \, E , L \, (\phi \, E) \rangle = \int  \left(  \frac{1}{2} - \mu \right) \, \phi^2 \, .
\end{align}
\end{proof}

When  $\Sigma$ has   $F$-index $I > 0$, then a   variation of Proposition \ref{p:extradims} gives   $\mu_{I+1} \geq \frac{1}{2}$.

\vskip1mm
Next, we adapt a result of Korevaar, \cite{K} (see \cite{GNY}; cf.  \cite{He}, \cite{YY}) to this setting:

\begin{Pro}	\label{p:kor}
There is a universal constant $C$ such that if $\Sigma^2\subset \RR^N$ is   closed, then  
\begin{align}
\mu_k (\Sigma)\,\lambda (\Sigma)\leq   C\,(1+\gamma)\,k\, .
\end{align}
\end{Pro}

\begin{proof}
Let $g$ be the metric on $\Sigma$ and define the conformal metric $g_1 = \e^{-f} \, g$. Let $dv_g$ and $dv_{g_1}$ be the corresponding area elements.
Note that $\lambda (\Sigma) = \frac{1}{4\pi} \, \Area_{g_1} (\Sigma)$.
Since $\e^{-f} \leq 1$, we have for any function $u$ that
\begin{align}
	\int |\nabla_g u|^2 \, \e^{-f} \, dv_g &\leq \int |\nabla_g u|^2 \, dv_g = \int |\nabla_{g_1} u|^2\, dv_{g_1}  \, , \\
	\int u^2 \, \e^{-f} \, dv_g & = \int u^2 \, dv_{g_1}  \, .
\end{align}
 Thus, for each $k$, it follows that
$
 	\mu_k = \mu_k (\cL) \leq \mu_k (\Delta_{g_1}) $.   Finally, \cite{K} gives that 
 \begin{align}
 	\mu_k (\Delta_{g_1}) \leq \frac{C (1+\gamma)\, k}{  \Area_{g_1} (\Sigma)} = \frac{ C (1+\gamma)\, k}{ 4\pi \, \lambda (\Sigma)} \, .
 \end{align}
\end{proof}

\begin{proof}[Proof of Theorem \ref{t:entropyextra2}]
Corollary $0.9$ in \cite{CM5} gives $C_1$ so that if $N \geq C_1 \, \lambda_{\Sigma}$, then there is a proper linear subspace $\cV \subset \RR^N$ so that $\Sigma \subset \cV$.  
Combining Propositions \ref{p:extradims} and \ref{p:kor} gives 
\begin{align}
	\frac{1}{2} \leq \mu_1 (\cL) \leq \frac{  C \, (1+\gamma)}{\lambda (\Sigma)}
	 \, .
\end{align}
The second claim follows from the first and  corollary $0.9$ in \cite{CM5}.      
\end{proof}

When $\Sigma$ is diffeomorphic to a sphere, we can argue as above and use \cite{He} to obtain \eqr{e:becomes}.

\begin{Con}
Theorem  \ref{t:entropyextra2} holds for complete $n$-dimensional $\lambda$-stable shrinkers.  
\end{Con}

\subsection{Spectrum of $\cL$ and $L$ for curves}

In \cite{AL}, Abresch-Langer classified  shrinking curves.  The   embedded ones are the circle and lines.  By lemma $6.16$ in \cite{CM5},
 every other   shrinking curve with $\lambda < \infty$  is closed, planar,  and strictly convex with Gauss map of degree at least two.

 \begin{Lem}	\label{l:shricur}
 If $\gamma \subset \RR^2$ is a closed shrinker and $\gamma \ne \SS^1_{\sqrt{2}}$, then:
 \begin{enumerate}
 \item The lowest eigenvalue of  $L$ is $-1$ and the next   is   less than $- \frac{1}{2}$.
 \item The lowest eigenvalue of $\cL$ is $0$ and the next is less than $\frac{1}{2}$.
 \end{enumerate}
 \end{Lem}

 \begin{proof}
 Let $\nn$ be the outward pointing unit normal.    Since $\gamma$ is strictly convex, $H = \langle \bH , \nn \rangle$ is positive.  By \cite{CM3}, $L \, H = H$ and, thus, $H$ is the lowest eigenfunction for $L$.
 
 Let $E_1 , E_2$ be the standard basis for $\RR^2$.
 By \cite{CM3}, the translation
 $u_i = \langle \nn , E_i \rangle$ is a $-\frac{1}{2}$-eigenfunction $L \, u_i = \frac{1}{2} \, u_i$.  
 Since $\nn$ is monotone as a map from $\SS^1$ to $\SS^1$ with degree at least two, $u_i$ has at least four  nodal domains.
 The Courant nodal domain theorem then gives that there must be another eigenvalue below $- \frac{1}{2}$.  This gives (1).
 
  For part (2), observe that $x_i = \langle E_i , x \rangle$ is a  $\frac{1}{2}$-eigenfunction $\cL \, x_i = -\frac{1}{2} \, x_i$. Since $\gamma$ is strictly convex and  $\nn$ has degree at least two, 
  $x_i$ has at least two positive local maxima on $\gamma$ and a negative local minimum between each maxima.  From this, we see that $x_i$ has at least four nodal domains and (2) now follows 
  from the   Courant nodal domain theorem.
 \end{proof}

 \begin{Cor}	\label{c:shricur}
 If $\gamma \subset \RR^2$ is a $F$-stable shrinker with $\lambda (\gamma) < \infty$, then $\gamma = \RR$ or $\gamma = \SS^1_{\sqrt{2}}$.
 \end{Cor}
 
 \begin{proof}
 We can assume  $\gamma$ is closed since otherwise $\lambda (\gamma) < \infty$ implies that $\gamma = \RR$.  
 If $\gamma \ne \SS^1_{\sqrt{2}}$, then Lemma \ref{l:shricur} gives an
  eigenvalue for $L$ strictly between $-1$ and $- \frac{1}{2}$.  The corresponding eigenfunction gives a negative variation that is orthogonal to $\bH$ and to translations.
 \end{proof}

 \section{Sharp bounds for the Gaussian Willmore functional}
 
 In general, $W$ is always bounded  by entropy (cf. corollary $3.34$ in \cite{CM3}):

\begin{Lem}	\label{l:UB}
If $\Sigma^n \subset \RR^N$, then	$2\, W(\Sigma)  \leq n \, \lambda (\Sigma)$.    Equality holds if and only if $\Sigma\subset \partial B_{\sqrt{2\,n}}$.   
\end{Lem}

\begin{proof}
Using that $4\, |\bH|^2 = |x|^2 - |x^T|^2$ and $\cL \, |x|^2 = 2\,n-|x|^2$, we get 
\begin{align}
	16\, \int_{\Sigma} |\bH|^2 \, \e^{-f}  = 8\, n \, \int_{\Sigma} \e^{-f} - 4 \, \int_{\Sigma} |x^T|^2 \, \e^{-f} =8\, n \, \int_{\Sigma} \e^{-f} -  \int_{\Sigma}(|x|^2 - 2\,n)^2 \, \e^{-f} \, .
\end{align}
\end{proof}

 In the rest of this section $\Sigma^n \subset \RR^N$ is   closed.   Given $\Sigma^2$  with genus $\gamma$, define   $C_{YY}$ by
\begin{align}
	C_{YY} = 
	\begin{cases}
	2 &{\text{ if }} \gamma = 0  \\
	\gamma+3&{\text{ if }} \gamma > 0
	\end{cases} \, .
	\end{align}

\begin{Thm}	\label{t:stableWillmore}
If   $\Sigma$ is a $F$-stable closed  surface with genus $\gamma$, then 
 \begin{align}	\label{e:stableWillmore}
	W(\Sigma) \leq  
	\begin{cases}
		  \frac{2\,C_{YY} }{\e}  & {\text{ if $\Sigma$ is oriented.}} \\
	  \frac{4\,C_{YY} }{\e} & {\text{ if $\Sigma$ is unoriented.}} 
	  \end{cases}
\end{align} 
\end{Thm}

 \vskip1mm
 Let $\Lambda$ be set of smooth functions $u$ on $\Sigma$ with $\int_{\Sigma} u \, |\bH|^2 \, \e^{-f} = 0$ and let $\Lambda^{\star} \subset \Lambda$ be the $u$'s with $\int_{\Sigma} u^2 \, |\bH|^2 \, \e^{-f} > 0$.
Define $\mu_{|\bH|^2} \geq 0$ by
 \begin{align}	\label{e:defmuH2}
 		\mu_{|\bH|^2} = \inf_{u \in \Lambda^{\star}} \, \, \frac{ \int_{\Sigma} |\nabla u|^2 \, |\bH|^2 \, \e^{-f} }{  \int_{\Sigma}  u^2 \, |\bH|^2 \, \e^{-f} } \, .
 \end{align} 
 When $|\bH| >0$,  then this infimum is achieved and $\mu_{|\bH|^2}$ is the first positive eigenvalue for the drift operator $\cL_{|\bH|^2} = \cL + \nabla_{\nabla \log |\bH|^2}$
  for the weight $\left| \bH \right|^2 \, \e^{-f}$.   

\begin{Lem}	\label{l:unstable12}
 If   $ \mu_{|\bH|^2} <  \frac{1}{2}$, then $\Sigma$ is $F$-unstable.
\end{Lem}

\begin{proof}
Since $ \mu_{|\bH|^2} <  \frac{1}{2}$, there exists a function $u$ with
\begin{align} 
	\int |\nabla u|^2 \, \left| \bH \right|^2 \, \e^{-f}  &< \frac{1}{2} \, \int u^2 \, \left| \bH \right|^2 \, \e^{-f} \, , \label{e:eigla} \\
	\int u \, \left| \bH \right|^2 \, \e^{-f} &= 0 \, .
\end{align}
The  equality gives  that $u \, \bH$ is $L^2$-orthogonal to $\bH$.  Using \eqr{e:phiHvar} and \eqr{e:eigla}   gives
\begin{align}	\label{e:delt21}
	(4\,\pi)^{\frac{n}{2}} \, \delta^2 ( u \, \bH) =   \int \left[ |\nabla u|^2 - u^2 \right] \, \left| \bH \right|^2 \, \e^{-f}<  - \frac{1}{2} \, \int u^2\, \left| \bH \right|^2 \, \e^{-f} \, .
\end{align}
Let $V^{\perp}$ with $V \in \RR^N$ be the  $L^2$-projection of $u\bH$ to the space of translations.   Since $\bH$ is 
orthogonal to $V^{\perp}$, it follows that $u \bH - V^{\perp}$ is orthogonal to both $\bH$ and translations.  We will show that $\delta^2 (u \bH - V^{\perp}) < 0$.  
Using   \eqr{e:2varr},  \eqr{e:1point5}, symmetry of $L$, and \eqr{e:delt21}, we have
\begin{align}	 
	(4\,\pi)^{\frac{n}{2}}  \, \delta^2 &(u \bH - V^{\perp})  = -\int \langle (u \bH - V^{\perp}), L\,(u \bH)    \rangle\,\e^{-f} = (4\,\pi)^{\frac{n}{2}} \, \delta^2 (u \bH) + \int \langle   V^{\perp} , L\,(u \bH)    \rangle\,\e^{-f}    \notag \\
		&= (4\,\pi)^{\frac{n}{2}} \, \delta^2 (u \bH) + \frac{1}{2} \,  \int \langle   V^{\perp} ,  u \bH     \rangle\,\e^{-f} <  - \frac{1}{2} \, \int u^2\, \left| \bH \right|^2 \, \e^{-f}  + \frac{1}{2} \,  \int  \left| V^{\perp}   \right|^2 \,\e^{-f} \, . 
\end{align}
Since $\| V^{\perp} \|_{L^2} \leq \| u \bH \|_{L^2}$, it follows that $\Sigma$ is $F$-unstable.
\end{proof} 

Similarly, we  define higher $\mu_{k,|\bH|^2}$'s to be the infimum over $k$-dimensional families in \eqr{e:defmuH2}.

\begin{Cor}
If $\Sigma^n\subset \RR^N$ is $F$-stable, then 
 $
\mu_{|\bH|^2,N+1} \geq 1$.
\end{Cor}

\begin{proof}
Suppose not.  Since the space of translations is $N$-dimensional, 
 we can find a function $\phi$   so that $\phi\,\bH$ is orthogonal to   translations (and $\bH$) and, moreover, 
\begin{align} 
	\int |\nabla \phi|^2 \, \left| \bH \right|^2 \, \e^{-f}  &<   \int \phi^2 \, \left| \bH \right|^2 \, \e^{-f} \, .
\end{align}
 Stability implies that $\delta^2(\phi\,\bH)\geq 0$ which  contradicts  this and  \eqr{e:phiHvar}.
\end{proof}

\begin{Lem}	\label{l:maxh}
For all $r>0$, we have
$
	r^2\,\e^{-\frac{r^2}{4}}\leq \frac{4}{\e}
$, with equality if and only if $r=2$.
\end{Lem}

\begin{proof}
Set $h(r)=r^2\,\e^{-\frac{r^2}{4}}$, then $h'(r)=2\,r\left(1-\frac{r^2}{4}\right)\,\e^{-\frac{r^2}{4}}$.   It follows that
$h(r)\leq h(2)$.    
\end{proof}

\begin{proof}[Proof of Theorem \ref{t:stableWillmore}]
We will assume first that $\Sigma$ is a topological sphere and roughly follow the argument of Hersch, \cite{He} (page $240$ in \cite{CM2}; cf. \cite{ChY}, \cite{CM1}).  Let  $g$ be the metric on $\Sigma$.     Since $\Sigma$ is a sphere, there is a  conformal diffeomorphism $\Phi : \Sigma \to \SS^2 \subset \RR^3$.  
The group of conformal transformations of $\SS^2$ contains a subgroup parametrized on the ball $B_1 \subset \RR^3$ with $z \in B_1$ corresponding to a ``dilation'' $\Psi_z$ in the direction of $\frac{z}{|z|}$ with $|z|$ determining the amount of the dilation (these are the $\psi (x,t)$'s on page 240 in \cite{CM1}).  As $|z| \to 1$, $\Psi_z$ takes $\SS^2 \setminus \{ \frac{-z}{|z|} \}$ to $\frac{z}{|z|}$.  Define a map $\cA: B_1 \to \RR^3$ by
\begin{align}
	\cA (z) = \frac{1}{\int_{\Sigma} |\bH|^2 \, \e^{-f}} \, \int_{\Sigma} \left( x_i \circ \Psi_z \circ \Phi \right) \,  |\bH|^2 \, \e^{-f} \, .
\end{align}
It follows that $\cA$ extends continuously to $\partial B_1 = \SS^2$ to be the identity on $\partial B_1$.  Elementary topology gives some $\bar{z} \in B_1$ so that $\cA (\bar{z}) = 0$.  Define $u_i$ on $\Sigma$ by
  $u_i = x_i \circ \Psi_{\bar{z}} \circ \Phi $ so that
 \begin{align}	\label{e:balanced}
 	\int_{\Sigma} u_i \,   |\bH|^2   \, \e^{-f} = 0 \, .
 \end{align}
Therefore, $F$-stability, \eqr{e:balanced} and Lemma \ref{l:unstable12} imply that for each $i$
\begin{align}
	 \int_{\Sigma}  u_i^2    \, |\bH|^2 \, \e^{-f}  \leq 2 \, \int_{\Sigma} |\nabla_g u_i |^2 \, |\bH|^2 \, \e^{-f}  \, .
\end{align}
Summing over $i$ and  using that $\sum_i u_i^2 \equiv 1$  gives
\begin{align}	\label{e:inhere}
	 4\pi \, W(\Sigma) = \int_{\Sigma}  |\bH|^2 \, \e^{-f} \leq  2 \, \sum_i \int_{\Sigma} |\nabla_g u_i |^2 \, |\bH|^2 \, \e^{-f}  \, .
\end{align}
Since $4\, |\bH|^2 \leq |x|^2$,  Lemma \ref{l:maxh} implies that  $|\bH|^2 \, \e^{-f}  \leq \e^{-1}$.  Using this in \eqr{e:inhere} and then conformal invariance of the energy    gives
\begin{align}	\label{e:inhere2}
	4\pi \, W(\Sigma)   \leq  \frac{2}{\e} \, \sum_i \int_{\Sigma} |\nabla_g u_i |^2  =  \frac{2}{\e} \, \sum_i \int_{\SS^2} |\nabla  x_i |^2   = \frac{16\, \pi}{\e}  
	  \, .
\end{align}
When $\Sigma$ is not a sphere, then we follow Yang-Yau, \cite{YY} (see also \cite{EI}, remark $1.2$ in \cite{Ka}), and replace $\Phi $ by a branched conformal map whose degree is bounded in the terms of $\gamma$.  The degree comes in as a factor in the equalities in \eqr{e:inhere2},  increasing the estimate for $W$.
\end{proof}

 \begin{proof}[Proof of Theorem \ref{t:stableWillmoreA}]
The case $\gamma =0$ of
 \eqr{e:stableWillmore}  gives the inequality.
 Suppose now that $\Sigma$ realizes equality.   First,  we must have equality in $4\, |\bH|^2 \leq |x|^2$ and, thus, $|x^T|^2 \equiv 0$ and $\Sigma$ is contained in a sphere.  We also get equality in 
 Lemma \ref{l:maxh} so this sphere has radius $2$.
 By Lemma \ref{p:basicH}, $\Sigma$ is minimal in $\partial B_{2} \subset \RR^N$.  Moreover, equality also implies that $\Sigma$ has 
   the same area as $\SS^2_2$  and, thus,  $\Sigma = \SS^2_2$ by Cheng-Li-Yau,  \cite{CgLY}.
    \end{proof}

\subsection{When $\mu_{|\bH|^2} = \frac{1}{2}$}

When we analyzed the case of equality in the bound for the Gaussian Willmore functional, one of the things that came out along the proof was that $\mu_{|\bH|^2} = \frac{1}{2}$ with 
multiplicity three and the eigenfunctions spanned the tangent space at each point.
We next analyze the borderline case where  $\cL_{|\bH|^2}$ has eigenvalue $\mu_{|\bH|^2} = \frac{1}{2}$ more generally.  Recall that the principal normal $ \bN = \frac{\bH}{|\bH|}$ is defined wherever $\bH \ne 0$.

\begin{Lem}	\label{l:h212}
If $\Sigma$ is $F$-stable and $\mu_{|\bH|^2} = \frac{1}{2}$, then for any eigenfunction $\phi$ of $\cL_{|\bH|^2}$ with eigenvalue $\frac{1}{2}$ there exists a vector $V \in \RR^N$ such that
$\phi\,\bH =V^{\perp}$  and
 $\nabla^{\perp}_{\nabla^T \phi}\bN=0$.
 \end{Lem}

\begin{proof}
Using \eqr{e:phiHvar},  integration by parts, and $\cL_{|\bH|^2} \, \phi = - \frac{1}{2} \, \phi$  gives
\begin{align}	 \label{e:thiabo}
	(4\,\pi)^{\frac{n}{2}} \, \delta^2 ( \phi \, \bH)     =   \int \left[ |\nabla \phi|^2 - \phi^2 \right] \, \left| \bH \right|^2 \, \e^{-f}=  - \frac{1}{2} \, \int \phi^2\, \left| \bH \right|^2 \, \e^{-f} \, .
\end{align}
Choose $V \in \RR^N$ so that $V^{\perp}$ is the $L^2$-projection of $\phi \, \bH$ to the space of translations.  
Since $V^{\perp}$ and $\phi\,\bH$ are orthogonal to $\bH$, it follows that $\phi \, \bH - V^{\perp}$ is orthogonal to both $\bH$ and translations. Thus, stability,  symmetry of $L$, $L \, V^{\perp} = \frac{1}{2} \, V^{\perp}$ and \eqr{e:thiabo} give
\begin{align}
	0& \leq (4\,\pi)^{\frac{n}{2}} \, \delta^2 (\phi \, \bH - V^{\perp})  = - \int_{\Sigma} \langle \phi \, \bH - V^{\perp} , L \, (\phi \, \bH)  \rangle \, \e^{-f} \notag \\
	&=- \frac{1}{2} \, \| \phi \, \bH \|_{L^2}^2 +   \frac{1}{2} \, \int_{\Sigma} \langle  V^{\perp} ,  \phi \, \bH   \rangle \, \e^{-f} =- \frac{1}{2} \, \| \phi \, \bH \|_{L^2}^2 +   \frac{1}{2} \,  \|  V^{\perp} \|_{L^2}^2
	  \, .
\end{align}
It follows that  $\| V^{\perp} \|_{L^2} = \| \phi \, \bH \|_{L^2}$ and, thus, that
 $\phi\,\bH=V^{\perp}$ and
 	$L\,(\phi\,\bH)=\frac{1}{2}\,\phi\,\bH$.  The second claim  follows from Leibniz' rule and $L\,(\phi\,\bH)=\frac{1}{2}\,\phi\,\bH$
\begin{align}
\frac{1}{2}\,\phi\,\bH &=L\,(\phi\,\bH) =\phi\,L\,\bH+(\cL\,\phi)\,\bH+ 2\, \nabla^{\perp}_{\nabla^T\phi}\bH\notag\\
&=(\phi+\cL_{|\bH|^2}\,\phi)\,\bH-|\bH|^{-2}\,\langle \nabla |\bH|^2,\nabla^T \phi\rangle\,\bH+ 2\, \nabla^{\perp}_{\nabla^T\phi}\bH =\frac{1}{2}\,\phi\,\bH  + 2\, |\bH|\,\nabla^{\perp}_{\nabla^T\phi}\bN \, .\notag
\end{align}
\end{proof}

\subsection{Frenet-Serret equations for shrinkers}

In $\RR^3$, the Frenet-Serret frame for a curve $\gamma$ parametrized by arclength is the orthonormal frame for $\RR^3$ along $\gamma$ consisting of the unit tangent $\gamma'$, the unit normal $\nn \equiv \frac{  \gamma''}{|{ \gamma''}|}$, and the binormal ${\bf{b}} \equiv \gamma' \times \nn$.    The Frenet-Serret formulas are ($k=| \gamma''|$):
\begin{align}
	\gamma'' &=   k\, \nn \, ,   \label{e:FS1} \\
	\nn' &= - k \, \gamma'    + \tau \, {\bf{b}} \, , \label{e:FS2}  \\
	{\bf{b}}' &= - \tau \, \nn \, ,   \label{e:FS3}
\end{align}
where $\tau$ is the torsion of $\gamma$.   
 We give an analog of these for  an oriented shrinker $\Sigma^n \subset \RR^{n+2}$.   
Let $J$ be the almost-complex structure of the (oriented) normal bundle.
Using $J$, we get a well-defined binormal   $\bB = J \, \bN$.   
Observe that $\langle \bB , x \rangle = \langle \bB , x^{\perp} \rangle = 2 \, \langle \bB , \bH \rangle = 0$, so that $\bB$ is always tangent to a sphere centered at $0$.  
We get the following
Frenet-Serret type formulas:
\begin{align}
	\nabla \bN &=     \tau \, \bB - \langle \bN , A(\cdot , \cdot) \rangle   \,  , \\
	\nabla \bB &= -\tau \, \bN - \langle \bB , A(\cdot , \cdot) \rangle 	\label{e:sfB2}
	\, .
\end{align}
It remains to compute the torsion.  Given a tangent vector $V$,  Lemma \ref{l:dH} gives
\begin{align}	\label{e:diffN}
	\tau (V) = \langle\nabla^{\perp}_V \, \bN,\bB\rangle = \langle\frac{ \nabla^{\perp}_V \, \bH }{|\bH|},\bB\rangle = 
		- \frac{1}{2} \, \langle\frac{ A (x^T , V)}{|\bH|},\bB\rangle \, .
\end{align}

\begin{Cor}   \label{c:frenet}
$\langle A , \bB \rangle = 0$ if and only if $\Sigma$ is a hypersurface in a hyperplane.
\end{Cor}

\begin{proof}
By \eqr{e:sfB2} and \eqr{e:diffN}, $\langle A , \bB \rangle = 0$ if and only if $\bB$ is a constant vector.   
\end{proof}

\begin{Thm}	\label{c:h212}
If $\Sigma^2 \subset \RR^4$ is $F$-stable, closed, oriented, and $\mu_{|\bH|^2} = \frac{1}{2}$, then $\Sigma^2 = \SS^2_2$.
 \end{Thm}

\begin{proof}  
We will show that $\langle A , {\bB} \rangle=0$ on an open set.  Once we have this,  Corollary \ref{c:frenet}  and unique continuation imply that $\Sigma$ is contained in a hyperplane and then \cite{CM1} and $F$-stability give that it is spherical or planar.
 Let $\phi$ be an eigenfunction as in   Lemma  \ref{l:h212}, so   $\phi \, \bH = V^{\perp}$ for  $V \in \RR^4$.
 Differentiating gives 
\begin{align}
	- \frac{\phi}{2} \, \langle A , {\bB} \rangle (x^T , \cdot) = \langle \nabla \, (\phi \, \bH) , \bB \rangle = \langle \nabla \, V^{\perp} , \bB \rangle = - \langle A , {\bB} \rangle (V^T , \cdot ) \, .
\end{align}
It follows that $  \frac{\phi}{2} \, x^T - V^T $ is in the kernel of $\langle A , {\bB} \rangle$ at each point.  If $\  \frac{\phi}{2} \, x^T - V^T $ vanishes everywhere, then so does $  \frac{\phi}{2} \, x - V  $ and, thus, 
$\phi \, x = 2\, V$ is constant. This is impossible, so there must be an open set $\Omega$ where $  \frac{\phi}{2} \, x^T - V^T  \ne 0$.  
However, the two by two matrix $\langle A , {\bB} \rangle$ is symmetric and trace-free, so it is either invertible or zero.  Since it has nontrivial kernel in $\Omega$, we see that $\langle A , {\bB} \rangle \equiv 0$ in $\Omega$.  This completes the proof.
\end{proof}

 \section{Entropy bounds}
 
 In this section we prove eigenvalue and entropy bounds without assuming that the dimension $N$ of the ambient Euclidean space is large compared with the entropy of the shrinker.   
 
 \begin{Thm}	\label{t:bdmu}
If $\Sigma^n \subset \RR^N$ is   $F$-stable  with finite entropy and $N \geq 2n$, then 
$
 \mu_{2\,n\,N}\geq \frac{1}{4}$.
 \end{Thm}

\begin{Cor}	\label{t:closedsurf}
There is a universal constant $C$ so that
if $\Sigma^2 \subset \RR^N$ is   closed and  $F$-stable   of genus $\gamma$, then 
$
\lambda (\Sigma) \leq C \, (1+\gamma)\,N$.
\end{Cor}

In the next lemma, $E_i$ is an orthonormal basis for $\RR^N$.

\begin{Lem}		\label{l:linearalg1}
If $\cV \subset \RR^N$ is an $n$-dimensional linear subspace, $\Pi$ and $\Pi^{\perp}$ are orthogonal projections to $\cV$ and $\cV^{\perp}$,   then for any  $k\in \ZZ $ with $1\leq k \leq N-n$
\begin{align}	\label{e:linearalg1}
	 k \leq \sum_{i=1}^{n+k} \,  \left| \Pi^{\perp} (E_i) \right|^2 \leq n+k\, .
\end{align}
\end{Lem}

\begin{proof}
Since
$ \sum_{i=1}^{N} \left| \Pi (E_i) \right|^2$ is the trace of a quadratic form, it is independent of the choice of basis.  Choosing the basis $\bar{E}_i$ so that $\bar{E}_1 , \dots , \bar{E}_n \in \cV$ and the rest are in $\cV^{\perp}$, we see that
\begin{align}
	\sum_{i=1}^{N} \, \left| \Pi (E_i) \right|^2 = \sum_{i=1}^{N} \, \left| \Pi (\bar{E}_i) \right|^2 = n \, .
\end{align}
Using this, we see that
\begin{align}
	\sum_{i=1}^{n+k} \, \left| \Pi^{\perp} (E_i) \right|^2 &= \sum_{i=1}^{n+k} \left( 1- \left| \Pi  (E_i) \right|^2 \right) = (n+k) - \sum_{i=1}^{n+k} \, \left| \Pi (E_i) \right|^2 
		 \geq (n+k) -  \sum_{i=1}^{N} \, \left| \Pi (E_i) \right|^2 = k \, . \notag
\end{align}
This gives the first inequality in \eqr{e:linearalg1}.
The second inequality is immediate. 
\end{proof}

\begin{Lem}	\label{p:Fstable}
Suppose that $\Sigma^n \subset \RR^N$ is $F$-stable, $k\in \ZZ $ and $1\leq k \leq N-n$.  If $\phi \in L^2$ is a function  so that 
\begin{align}	\label{e:ortho1}
	\int_{\Sigma} \phi \, \langle E_j^{\perp}  , \bH \rangle\, \e^{-f} &= 0 {\text{ for  $j=1, \dots , n+k$}} , \\
	\int_{\Sigma} \phi \, \langle E_j^{\perp}  , E_{\ell}^{\perp} \rangle\, \e^{-f} &= 0 {\text{ for  $(j,\ell) \in \{ 1 , \dots , n+k \} \times \{ 1, \dots , N\}$}}  \, , \label{e:ortho2}
\end{align}
then  $\int_{\Sigma} \phi^2 \, \e^{-f} \leq 2\,\left(\frac{n}{k}+1\right) \, \int_{\Sigma} |\nabla \phi |^2 \, \e^{-f}$.
\end{Lem}

\begin{proof}
By \eqr{e:ortho1} and \eqr{e:ortho2}, the vector field $\phi \, E_j^{\perp}$ is orthogonal to $\bH$ and to translations for each $j=1, \dots , n+k$.  By the definition of $F$-stability and Lemma \ref{l:phiVvar}
\begin{align}	\label{e:fromLem}
	0 \leq (4\,\pi)^{\frac{n}{2}} \, \delta^2 (\phi \, E_j^{\perp}) = \int \left[ |\nabla \phi |^2 -  \frac{1}{2} \, \phi^2 \right] \, \left| E_j^{\perp} \right|^2 \, \e^{-f}  \, .
\end{align}
Finally, we sum \eqr{e:fromLem} over $j\leq n+k$ and use 
  $k\leq \sum_{j=1}^{n+k}  \left| E_j^{\perp} \right|^2 \leq n+k$  by Lemma \ref{l:linearalg1}.
\end{proof}

\begin{Cor}	\label{t:bdmu2}
If $\Sigma^n \subset \RR^N$ has $F$-index $I$, then for any  $k \in \ZZ$ with $1\leq  k\leq N-n$ 
\begin{align}
 \mu_{(n+k)\,(N+I-\frac{1}{2}\,(n+k-3))}\geq \frac{k}{2\,(n+k)}\, .
\end{align}
\end{Cor}

\begin{proof} 
We will first show the corollary when $I=0$.    For a fixed $\phi$, \eqr{e:ortho1} and \eqr{e:ortho2} give $n+k$ and $\frac{1}{2}\,(n+k-1)\,(n+k)+(n+k)+ (n+k) \, (N-(n+k))=(n+k) \, (N-\frac{1}{2}\,(n+k-1))$ homogeneous linear equations.  So $(n+k)\,(N-\frac{1}{2}\,(n+k-3))$ linear equations.  Thus, we can choose a  linear combination $\phi = \sum a_i \,u_i$ of the functions
\begin{align}
	u_0 , u_1 , \dots , u_{(n+k)\,(N-\frac{1}{2}\,(n+k-3))} 
\end{align}
with $\sum a_i^2 =1$ and so $\phi$  satisfies \eqr{e:ortho1} and \eqr{e:ortho2}.   Lemma \ref{p:Fstable} with this $\phi$ gives
\begin{align}
	1 &= \int \phi^2 \, \e^{-f} \leq 2\,\left(\frac{n}{k}+1\right)\,\int |\nabla \phi |^2 \, \e^{-f} = 2\,\left(\frac{n}{k}+1\right)\,\sum \left( a_i^2 \, \mu_i \right) \notag\\
	&\leq 2\,\left(\frac{n}{k}+1\right)\,\mu_{(n+k)\,(N-\frac{1}{2}\,(n+k-3))} \, \sum a_i^2 = 2\,\left(\frac{n}{k}+1\right)\,\mu_{(n+k)\,(N-\frac{1}{2}\,(n+k-3))} \, .
\end{align}
The case where $I>0$ follows with obvious modifications.  
\end{proof}

Specializing to $k=n$ and $I=0$ gives Theorem \ref{t:bdmu}.

\begin{proof}[Proof of Corollary \ref{t:closedsurf}]
Corollary \ref{t:bdmu2} gives for $n=2$ that $\mu_{2\,(2\,N-1)}\geq \frac{1}{4}$.   Combining this with Proposition \ref{p:kor} gives
  the corollary.
\end{proof}

Corollary \ref{t:closedsurf} extends easily to give general entropy bounds in terms of 
  the index $I>0$.  

\begin{Con}
There exist $\alpha < 1$ and $C_{\alpha} = C_{\alpha} (\alpha , \gamma)$ so that if $\Sigma^2 \subset \RR^N$, then the multiplicity of the $\frac{1}{2}$ eigenvalue for $\cL$ is at most
$C_{\alpha} \, \lambda^{\alpha} (\Sigma)$.    If so, then \cite{CM5} would give Theorem \ref{t:entropyextra2} without the assumption
  $N \geq C \, \lambda (\Sigma)$.
\end{Con}

\end{document}